\documentclass[11pt]{amsart}
\usepackage[margin=1.4in]{geometry}
\usepackage{amssymb}
\usepackage{relsize}
\usepackage{ wasysym }
\usepackage{comment}

\usepackage{xcolor}

\newtheorem{thm}{Theorem}[section]
\newtheorem{cor}[thm]{Corollary}

\newtheorem{remark}[thm]{Remark}

\newtheorem{prop}[thm]{Proposition}

\newtheorem{exam}[thm]{Example}
\newtheorem{defn}[thm]{Definition}

\numberwithin{equation}{section}

\newcommand{\bb}[1]{\mathbb{#1}}
\newcommand{\cl}[1]{\mathcal{#1}}

\newcommand{\cW}{\mathcal{W}}

\newcommand{\Wmax}[1]{\mathcal{W}^{#1\text{-}\mathrm{max}}}
\newcommand{\Wmin}[1]{\mathcal{W}^{#1\text{-}\mathrm{min}}}
\newcommand{\smboxplus}{\text{\footnotesize{$\boxplus$}}}
\newcommand{\smboxplusraised}{\raisebox{0.01 in}{\text{\footnotesize{$\boxplus$}}}}

\begin{document}

\title[Matrix Range Characterizations of Operator System Properties]{Matrix Range Characterizations of Operator System Properties}
\author[B. Passer]{Benjamin Passer}
\address{Department of Pure Mathematics, University of Waterloo,
Waterloo, ON,  N2L 3G1, Canada}
\email{passbend@gmail.com}
\author[V.~I.~Paulsen]{Vern I.~Paulsen}
\address{Institute for Quantum Computing and Department of Pure Mathematics, University of Waterloo,
Waterloo, ON,  N2L 3G1, Canada}
\email{vpaulsen@uwaterloo.ca}

%\thanks{}
%\keywords{frame}
%\subjclass[2000]{Primary 46L15; Secondary 47L25}

\begin{abstract} For finite-dimensional operator systems $\mathcal{S}_{\sf T}$, ${\sf T} \in B({\cl H})^d$, we show that the local lifting property and $1$-exactness of $\mathcal{S}_{\sf T}$ may be characterized by measurements of the disparity between the matrix range $\cW({\sf T})$ and the minimal/maximal matrix convex sets over its individual levels. We then examine these concepts from the point of view of free spectrahedra, direct sums of operator systems, and products of matrix convex sets.
\end{abstract}

\maketitle

%%%%%%%%%%%%%%%%%%%%%%%%%%%%%%%%%%%%%%%%%%%%%%%%%%%%%%%%%%%%%%%%%%%%%%%%%%%%
%%%
\section{Introduction} 

Since operator systems give rise to matrix convex sets, and vice versa, it is possible to view results from one category in terms of the other. In this paper, we seek to unify these points of view somewhat. We first highlight some overlap in the theories, then discuss how inclusion constants of matrix convex sets relate to abstract operator system properties, such as local lifting and $1$-exactness. Finally, we point out how products of matrix convex sets, similarly direct sums of operator systems, interact with these constants.

Given a $d$-tuple ${\sf T} = (T_1, \ldots, T_d) \in B({\cl H})^d$ of bounded operators on a Hilbert space $H$, one may consider the finite-dimensional operator system
\[ \cl S_{\sf T} := \text{span} \{ I, T_1, \ldots ,T_d, T_1^*, \ldots  T_d^* \}  \]
which the operators $T_i$ generate inside $B(\cl H)$. Similarly, the {\bf (joint) matrix range} of ${\sf T}$ is defined as the matrix convex set
\[ \cW({\sf T}) = \bigcup_{n=1}^\infty \cW_n({\sf T}), \] 
where 
\[ \cW_n({\sf T}) = \{ (\phi(T_1), \ldots, \phi(T_d)): \phi: \cl S_{\sf T} \to M_n \text{ is UCP} \}\]
denotes the collection of unital completely positive images of $\sf T$ into the set of $n \times n$ complex matrices. Hence, we regard $\cW_n({\sf T})$ as a subset of the vector space of $d$-tuples of $n \times n$ matrices $M_n^d$ and $\cW({\sf T})$ as a subset of the disjoint union $\bigcup\limits_{n=1}^\infty M_n^d$, which we denote as $\mathbb{M}^d$.

By Arveson's extension theorem \cite[Theorem 1.2.3]{Arv_sub}, any UCP map $\cl S_{\sf T} \to M_n$ extends to a UCP map $B( \cl H) \to M_n$, so it also holds that
\[ \cW_n({\sf T}) = \{ (\phi(T_1), \ldots, \phi(T_d)): \phi: \cl B( \cl H) \to M_n \text{ is UCP} \}.\]
The matrix range $\cW({\sf T})$ is an example of a closed and bounded matrix convex set, and moreover, every closed and bounded matrix convex set over $\mathbb{C}^d$ is of the form $\cW({\sf T})$ for some tuple ${\sf T}$ of bounded operators by \cite[Proposition 3.5]{DDSS}. 

Existence of UCP maps between finite-dimensional operator systems (with specified bases) is encoded in the matrix range. Given $\cl S_{\sf T}$ and $\cl S_{\sf R}$ of the same dimension, the map $\Psi: \cl S_{\sf T} \to \cl S_{\sf R}$ given by 
\[\Psi(I) = I, \hspace{.2 in} \Psi(T_i) = R_i,  \hspace{.2 in} \Psi(T_i^*) = R_i^*\]
is UCP if and only if $\cW({\sf R}) \subseteq \cW({\sf T})$ by \cite[Theorem 5.1]{DDSS}. Further, $\Psi$ is a complete order isomorphism if and only if $\cW({\sf T}) = \cW({\sf R})$.

Considering $\cW({\sf T})$ in place of $\mathcal{S}_{\sf T}$ allows one to use the Hausdorff distance in the matrix norm to derive consequences about UCP maps. For example, in the case $d = 1$, this point of view is pursued in \cite{Pa-1982} in order to connect approximation properties to the Smith-Ward problem. Similarly, one may consider a related scaling problem: what is the smallest $C > 0$ so that $\cW({\sf R}) \subseteq C \, \cW({\sf T})$, or equivalently that the map $\Psi: \mathcal{S}_{\sf T} \to \mathcal{S}_{\sf R}$ given by
\[ \Psi(I) = I, \hspace{.2 in} \Psi(T_i) = \frac{1}{C} \, R_i, \hspace{.2 in} \Psi(T_i^*) = \frac{1}{C} \, R_i^* \]
is UCP? This problem is considered in \cite{DDSS,Pas-Sh-So, Pas_SSM} for matrix ranges arising in different contexts.  Connections between these ideas and essential matrix ranges are explored in \cite{Li_essMR}.

A common concern which connects the above works is the disparity between $\cW({\sf T})$ and the smallest and largest matrix convex sets generated over $\cW_k({\sf T})$. These objects correspond to the earlier concepts of $k$-minimal and $k$-maximal operator system structures from \cite{Xh-phd, Xh-JFA}, and they were subsequently examined in \cite{Kriel} from the matrix convex point of view (generalizing results in \cite{DDSS} for $k = 1$). In this manuscript, we will connect the approximation of $\cW({\sf T})$ by these sets to operator system properties, namely the lifting property and $1$-exactness. Our main result is as follows.

\begin{thm}\label{thm:main}
Let ${\sf T} \in B(\cl H)^d$. Then $\mathcal{S}_{\sf T}$ has the lifting property if and only if the maximal matrix convex sets over $\cW_k({\sf T})$ converge to $\cW({\sf T})$ in the Hausdorff distance as $k \to +\infty$. Similarly, $\mathcal{S}_{\sf T}$ is $1$-exact if and only if the minimal matrix convex sets over $\cW_k({\sf T})$ converge to $\cW({\sf T})$ in the Hausdorff distance as $k \to +\infty$.
\end{thm}

We find that to prove this theorem, it is useful to consider both Hausdorff distances and scaling properties (but note that our notion of Hausdorff approximation is distinct from that of \cite{Ge-random}, as we require convergence to be uniform across all the levels). Moreover, in doing so, we relate these approximations to free spectrahedra and their polar duals. In section \ref{sec:review}, we briefly review the $k$-minimal/$k$-maximal operator systems and matrix convex sets, and we set the notation we will use for the rest of the manuscript. Section \ref{sec:scale_prop} then develops the proof of Theorem \ref{thm:main} and gives extensions and examples. Finally, section \ref{sec:products} discusses how the scales considered in previous sections behave under products of matrix convex sets (equivalently, direct sums of operator systems).

%%%%%%%%%%%
%%%%%%%%%
\section{k-minimality and k-maximality}\label{sec:review}

Given an operator system $\cl S$ and a natural number $k$, Xhabli \cite{Xh-phd, Xh-JFA} introduced two new operator systems on the ordered vector space $\cl S$, denoted $\text{OMIN}_k(\cl S)$ and $\text{OMAX}_k(\cl S)$. To specify these operator systems, we need to describe the positive cones at each matrix level.

The positive cone $M_n(\text{OMIN}_k(\cl S))^+$ is denoted $C_n^{k\text{-min}}$ and is defined by
\[ C_n^{k\text{-min}} = \{ (x_{i,j} ) \in M_n(\cl S):  (\phi(x_{i,j})) \ge 0 \text{ for every UCP } \phi: \cl S \to M_k \}.\]
This operator system is designed for the following universal property: if $\cl T$ is an operator system, then 
\begin{equation} \psi: \cl T \to \cl S \hspace{.05 in} k\text{-positive} \hspace{.05 in} \iff \hspace{.05 in}  \psi: \cl T \to \text{OMIN}_k(\cl S) \text{ completely positive}.\end{equation}

On the other hand, the positive cone $M_n(\text{OMAX}_k(\cl S))^+$ is denoted $C_n^{k\text{-max}}$ and is defined in two stages. First, one sets
\[ D_n^{k\text{-max}} = \{ A^*DA : \exists m, A \in M_{mk, n},  D= \text{diag}( D_1, \ldots, D_m), \text{ with } D_i \in M_k(\cl S)^+, \forall i \}. \]
These sets are matrix convex, but they do not necessarily satisfy an Archimedean axiom. The sets $C_n^{k\text{-max}}$ are then obtained through Archimedeanization. That is,
\[ C_n^{k\text{-max}} = \{ (x_{i,j}) \in M_n(\cl S) : (x_{i,j}) + \varepsilon I_n \in D_n^{k\text{-max}}, \forall \varepsilon > 0 \},\]
where $I_n$ denotes the diagonal matrix whose entries are all equal to the identity of $\cl S$.

The corresponding objects in finite-dimensional closed matrix convex sets are found in \cite[Definition 3.1]{Kriel}, as $\text{kz}_k$ and $\text{pz}_k$. In this manuscript, we will only consider the bounded case, and we caution the reader that the presentation given here critically uses this assumption. We also use notation consistent with that of \cite{DDSS}.

 If $\mathcal{C} \subseteq \mathbb{M}^d$ is closed, bounded, and matrix convex, then we set
\[ \Wmin{k}(\mathcal{C}) = \text{kz}_k(\mathcal{C}) = \bigcap\{ \mathcal{L}:   \mathcal{L} \text{ is matrix convex with } \mathcal{L}_k = \mathcal{C}_k\}. \]
Certainly $\Wmin{k}(\mathcal{C}) \subseteq \mathcal{C}$ remains bounded, and closedness follows from \cite[Lemma 1.14]{Kriel}. Specifically, $\Wmin{k}(\mathcal{C})$ consists of matrix convex combinations $Y = \sum V_i^* X^{(i)} V_i$ of points in $\mathcal{C}_k$, and any such combination living in level $n$ may be written with a fixed number of summands. Then, since $\mathcal{C}_k$ is compact, any sequence of such combinations in level $n$ admits limit points in each of the terms. This guarantees that a limit point of $\Wmin{k}(\mathcal{C})$ remains in $\Wmin{k}(\mathcal{C})$.

Similarly, under the same assumptions, we set 
\[ \Wmax{k}(\mathcal{C}) = \text{pz}_k(\mathcal{C}) = \{ X \in \mathbb{M}^d: \cW_k(X) \subseteq \mathcal{C}_k\}. \]
A matrix convex set is bounded if and only if its first level (or equivalently, any particular level $k$) is bounded, so $\Wmax{k}(\mathcal{C})$ is bounded. Similarly, $\Wmax{k}(\mathcal{C})$ is closed because both $\mathcal{C}_k$ and the matrix state space are compact.

The above operations generalize $\mathcal{W}^{\text{min}}$ and $\mathcal{W}^{\text{max}}$ of \cite{DDSS}, up to a slight change in notation. That is, $\mathcal{W}^{\text{min}}$ and $\mathcal{W}^{\text{max}}$ receive closed convex subsets of Euclidean space as inputs, whereas the above operations receive matrix convex sets. However, $\Wmin{k}$ and $\Wmax{k}$ only care about the $k$th level of the input set, so we use the following convention. If necessary, $\Wmin{k}$ and $\Wmax{k}$ may also receive $C^*$-convex subsets of $M_k^d$. In this case, the output is identical to the result of applying $\Wmin{k}$ or $\Wmax{k}$ to any matrix convex set with the prescribed level $k$:
\begin{equation}\label{eq:k_or_full} \Wmin{k}(\mathcal{C}_k) = \Wmin{k}(\mathcal{C}), \hspace{0.45 in} \Wmax{k}(\mathcal{C}_k) = \Wmax{k}(\mathcal{C}).  \end{equation}
Once this convention is applied, it follows that $\mathcal{W}^{\text{min}}$ and $\mathcal{W}^{\text{max}}$ do indeed correspond to $\Wmin{1}$ and $\Wmax{1}$.

Recall that the polar dual of a closed and bounded matrix convex set $\mathcal{C}$, denoted $\mathcal{C}^\circ$, is the collection of matrix tuples $A$ such that
\[ \text{Re}\left( \sum\limits_{k=1}^d A_k \otimes X_k \right) \leq I \]
for every $X \in \mathcal{C}$. One may similarly define a polar dual for matrix convex sets of self-adjoints \cite[\S 3]{DDSS}. It is well-known that the polar dual of $\mathcal{C}$ is itself closed and bounded if and only if $0$ is in the interior of $\mathcal{C}_1$ (see \cite[Remark 2.5]{Kriel}, for example). In this case, the polar dual reverses the roles of $\Wmin{k}$ and $\Wmax{k}$ as a consequence of \cite[Lemma 3.7]{Kriel}. That is, if $0$ is in the interior of a closed and bounded matrix convex set $\mathcal{C}$, then
\begin{equation}\label{eq:polar_dual_min_max} \Wmax{k}(\mathcal{C})^{\circ} = \Wmin{k}(\mathcal{C}^\circ), \hspace{.45 in} \Wmin{k}(\mathcal{C})^\circ = \Wmax{k}(\mathcal{C}^\circ). \end{equation}
Note that the above presentation appears to be simpler than \cite[Lemma 3.7]{Kriel} only because we have restricted to the case where both $\mathcal{C}$ and $\mathcal{C}^{\circ}$ are closed and bounded. 

In the same vein, the earlier result \cite[Theorem 9.9]{Ka-nuc} shows that operator system duality switches the roles of $\text{OMIN}_k$ and $\text{OMAX}_k$. This is no coincidence, as the polar dual for a matrix convex set $\cW({\sf T})$ corresponds to a particular choice of basis and Archimedean order unit for the dual operator system of $\mathcal{S}_{\sf T}$. To see this, let ${\sf T} =(T_1, \ldots,T_d) \in B(\cl H)^d$ and let $\cl S_{\sf T} = \text{span} \{ I, T_1, \ldots, T_d, T_1^*, \ldots ,T_d^* \}$. For convenience, set $T_{d+k} = T_k^*$ for $1 \le k \le d$ and $I= T_0$.
If  $\dim(\cl S_T) = 2d+1$, then we have dual functionals, $\delta_i,  0 \le i \le 2d$ where $\delta_i(T_j) = \delta_{i,j}$. The dual of the operator system $\cl S_{\sf T},$ which we denote $\cl S_{\sf T}^\prime$, is a matrix ordered space, but since we are in the finite dimensional case, it is also an operator system as soon as we pick an appropriate matrix order unit.   Also, recall that given $f \in \cl S_T^\prime$, the *-operation is defined by
\[ f^*(X) = \overline{f(X^*)}.\] It is easily checked that $\delta_0^* = \delta_0$ and that for $1 \le i \le d$, $\delta_i^* = \delta_{d+i}$.   Thus,   $\cl S_{\sf T}^\prime = \text{span} \{ \delta_0, \delta_1, \ldots, \delta_d, \delta_1^*,\ldots, \delta_d^* \}$. The following is essentially a restatement of \cite[Corollary 6.3.12]{Adam-thesis}.

\begin{prop}\label{prop:polar_opsys_dual} Let ${\sf T} =(T_1, \ldots, T_d) \in B(\cl H)^d$ and assume that $0$ is in the interior of $\mathcal{W}_1({\sf T})$.  Then
\begin{itemize}
\item $\dim(\cl S_{\sf T}) = 2d+1$,
\item $\delta_0$ is an Archimedean order unit for the matrix ordered space $\cl S_T^\prime$, so that $(\cl S_{\sf T}^\prime, \delta_0)$ is an abstract operator system, and
\item when we regard ${\sf R}=( \delta_1, \ldots, \delta_d)$ as a $d$-tuple in this operator system, we have that
\[ \cl W({\sf R}) = \cl W({\sf T})^\circ.\]
\end{itemize}
\end{prop}

The above claim can be similarly adjusted to even-dimensional operator systems, or to operator systems with a basis of self-adjoints. In particular, from \cite[Corollary 6.3.10]{Adam-thesis}, every finite-dimensional operator system admits a basis of self-adjoints $1, P_1, \ldots, P_n$ such that $0$ is in the interior of $\mathcal{W}_1({\sf P}) \subseteq \mathbb{R}^n$. For even-dimensional operator systems, one can choose a basis $1, T_1, \ldots, T_d, T_1^*, \ldots, T_{d-1}^*$ where $T_d = T_d^*$, but the other $T_i$ are non self-adjoint. We can always assume that we are in a similar situation as in the proposition with a slight abuse of notation, e.g. by viewing the interior of a set in $\mathbb{R}^n$ or $\mathbb{C}^{d-1} \oplus \mathbb{R}$ instead of $\mathbb{C}^d$. Similarly, one may consider the self-adjoint polar dual (see \cite[\S 3]{DDSS}) in place of the polar dual, when needed, and we leave these details to the reader. We will refer to all such changes of a spanning set or basis as {\bf recoordinatizing} the operator system. Except for the requirement that $0$ is an interior point, when appropriate, our results are generally independent of the coordinate system.

We will consider the approximation of $\mathcal{C} = \cW({\sf T})$ by $\Wmin{k}(\mathcal{C})$ and $\Wmax{k}(\mathcal{C})$ in both a Hausdorff distance sense and a scaling sense. Given two matrix tuples ${\sf A}, {\sf B} \in M_n^d$ of the same dimension, we set
\[ ||{\sf A} - {\sf B}|| = \sum\limits_{j=1}^d ||A_i - B_i||, \]
and similarly for tuples with elements in any $C^*$-algebra. We then define the Hausdorff distance for subsets of $M_n^d$ in the usual way, which we extend to matrix convex sets by setting $\text{dist}(\mathcal{C}, \mathcal{D}) := \sup_{n} \text{dist}(\mathcal{C}_n, \mathcal{D}_n)$. Note in particular that convergence in this Hausdorff distance requires a uniform approximation across all levels, in contrast with the Hausdorff approximation used in \cite{Ge-random}. Similar ideas give rise to the following constants.

\begin{defn}
Given a closed and bounded matrix convex set $\mathcal{C} \subseteq \mathbb{M}^d$, let
\begin{itemize}
\item $\alpha_k(\mathcal{C}) = \inf \{ r \geq 1:  \Wmax{k}(\mathcal{C}) \subseteq r \cdot  \Wmin{k}(\mathcal{C})\}$,
\item $\beta_k(\mathcal{C}) = \inf \{ r \geq 1:  \mathcal{C}  \subseteq r \cdot  \Wmin{k}(\mathcal{C}) \}$, and
\item $\gamma_k(\mathcal{C}) = \inf \{ r \geq 1: \Wmax{k}(\mathcal{C}) \subseteq r \cdot \mathcal{C} \},$
\end{itemize}
with the convention that $\inf(\varnothing) = +\infty$. If $\mathcal{C}$ is equal to $\cW({\sf T})$, we also denote these constants as $\alpha_k({\sf T})$, $\beta_k({\sf T})$, and $\gamma_k({\sf T})$, respectively. 
\end{defn}

Note that the constant $\theta$ of \cite[(3.7)]{Pas-Sh-So} is just $\alpha_1$ with a slight change in notation for the input set. Moreover, it is immediate that 
\begin{equation}\label{eq:trivial_bounds} \max\{ \beta_k(\mathcal{C}), \gamma_k(\mathcal{C})\} \leq \alpha_k(\mathcal{C}) \le \beta_k(\mathcal{C}) \gamma_k(\mathcal{C}) \end{equation}
and that each parameter is a non-increasing function of $k$. Thus, we may also set
\begin{equation} \alpha(\mathcal{C}) = \lim_{k \to +\infty} \alpha_k(\mathcal{C}), \hspace{.3 in} \beta(\mathcal{C}) = \lim_{k \to +\infty} \beta_k(\mathcal{C}), \hspace{.3 in} \gamma(\mathcal{C}) = \lim_{k \to +\infty} \gamma_k(\mathcal{C}), \end{equation}
with a similar notational convention as above if $\mathcal{C} = \cW({\sf T})$. In this case, we also let ${\sf T}^{k\text{-min}}$ and ${\sf T}^{k\text{-max}}$ denote a (non-unique) operator tuple with 
\begin{equation} \Wmin{k}(\cW({\sf T})) = \cW({\sf T}^{k\text{-min}}), \hspace{.3 in} \Wmax{k}(\cW({\sf T})) = \cW({\sf T}^{k\text{-max}}). \end{equation}
That is,
\begin{equation}\label{eq:equality_at_n} 1 \leq n \leq k \hspace{.2 in} \implies \hspace{.2 in} \cW_n({\sf T}^{k\text{-min}}) = \cW_n({\sf T}) = \cW_n({\sf T}^{k\text{-max}}), \end{equation}
and $\cW({\sf T}^{k\text{-min}})$ and $\cW({\sf T}^{k\text{-max}})$ are extremal matrix convex sets with this property. Equivalently, ${\sf T}^{k\text{-min}}$ spans $\text{OMIN}_k({\cl S}_{\sf T})$ and ${\sf T}^{k\text{-max}}$ spans $\text{OMAX}_k({\cl S}_{\sf T})$.

The constants above are not solely determined by the operator system $\mathcal{S}_{\sf T}$, in that they also encode information about the choice of coordinates. For the most part, this lack of uniqueness will not be a problem, but we will generally consider the assumption that $0$ is in the interior of $\cW_1({\sf T})$ in order to show that the constants are finite.

\begin{prop}\label{prop:scale_position}
If ${\sf T} \in B(\cl H)^d$ and $0$ is in the interior of $\cW_1({\sf T})$ (viewed in real Euclidean space if needed), then there is a uniform bound on $\alpha_k({\sf T})$, $\beta_k({\sf T})$, and $\gamma_k({\sf T})$. 
\end{prop}
\begin{proof}
Given (\ref{eq:trivial_bounds}) and the fact that each expression is non-increasing in $k$, it suffices to bound $\alpha_1({\sf T})$. This bound is known as an immediate consequence of \cite[\S 7]{DDSS}, which gives that the real Euclidean ball $\mathbb{B}^{n}$ has $\Wmax{1}(\mathbb{B}^{n}) \subseteq n \, \Wmin{1}(\mathbb{B}^{n})$. Hence, if $\varepsilon \, \mathbb{B}^{n} \subseteq \cW_1({\sf T}) \subseteq M \, \mathbb{B}^{n}$, it also follows that $\Wmax{1}(\cW({\sf T})) \subseteq \frac{M}{n \varepsilon} \, \Wmin{1}(\cW({\sf T}))$, and therefore $\alpha_1({\sf T}) \leq \frac{M}{n \varepsilon}$. 
\end{proof}

Every finite-dimensional operator system admits a presentation as $\mathcal{S}_{\sf T}$ where $0$ is in the interior of $\cW_1({\sf T})$, so the assumption of Proposition \ref{prop:scale_position} is not particularly demanding. However, we are most interested in the case where the limiting constants are exactly $1$. This, in turn, gives estimates on the Hausdorff distance. The conversion of a scaled comparison
\[ a \, \mathcal{D} \subseteq \mathcal{C} \subseteq b \, \mathcal{D} \]
into a Hausdorff distance estimate is trivial given bounds on all elements of $\mathcal{C}$ and $\mathcal{D}$, and a reverse estimate follows with some control on the placement of zero. These arguments work for matrix convex sets in much the same way as for classical convex sets, modulo some slight complications from the existence of multiple matrix levels.

\begin{prop}\label{prop:precise_dist}
Let $\mathcal{C}$ and $\mathcal{D}$ be matrix convex sets over $\mathbb{R}^d$ (or over $\mathbb{C}^n$ where $d = 2n$, etc.). If $M > 0$ is a fixed, uniform bound on all elements of $\mathcal{C} \cup \mathcal{D}$, then it follows that for positive constants $a$ and $b$,
\[ a \, \mathcal{D} \subseteq \mathcal{C} \subseteq b \, \mathcal{D} \,\,\,\, \implies \,\,\,\,\, \emph{dist}(\mathcal{C}, \mathcal{D}) \leq M \, \max\{|a^{-1} - 1|, |b - 1|\}. \]
Similarly, if $\delta > 0$ is a fixed constant such that the Euclidean ball satisfies $\mathbb{B}_\delta \subseteq \mathcal{C}_1 \cap \mathcal{D}_1$, then it follows that for any $\varepsilon > 0$, 
\[ \emph{dist}(\mathcal{C}, \mathcal{D}) < \varepsilon \,\,\,\, \implies \,\,\,\,\, \cfrac{1}{a} \, \mathcal{D} \subseteq \mathcal{C} \subseteq a \, \mathcal{D}, \]
where $a = 1 + d \varepsilon/\delta$.
\end{prop}
\begin{proof}
The first claim is left to the reader, as it does not use matrix convexity. Suppose $\delta, \varepsilon > 0$ are such that $\mathbb{B}_\delta \subseteq \mathcal{C}_1 \cap \mathcal{D}_1$ and $\text{dist}(\mathcal{C}, \mathcal{D}) < \varepsilon$. Given an arbitrary ${\sf C} \in \mathcal{C}$, we may approximate ${\sf C}$ within $\varepsilon$ by an element of $\mathcal{D}$, equivalently
\[ {\sf C} = {\sf D} + \frac{d \varepsilon}{\delta} \, {\sf E}, \,\,\,\,\,\,\, {\sf D} \in \mathcal{D}, \,\, ||{\sf E}|| < \frac{\delta}{d} \,. \]
Now, from $||{\sf E}|| < \delta/d$, we may certainly conclude that $\mathcal{W}_1({\sf E})$ is contained in $\mathbb{B}_{\delta/d}$. Applying the dilation results of \cite[\S 7]{DDSS} in $d$ self-adjoint variables shows that ${\sf E}$ is in the minimal matrix convex set over $\mathbb{B}_\delta$. In particular, ${\sf E} \in \mathcal{D}$. Finally, we may write ${\sf C}$ as
\[ {\sf C} = \left( 1 + d \, \varepsilon/\delta \right) \, \left( \cfrac{1}{1+ d\varepsilon / \delta} \, {\sf D} + \cfrac{d \varepsilon / \delta}{1 + d \varepsilon / \delta} \, {\sf E} \right), \]
so that ${\sf C}$ is $1 + \frac{d \, \varepsilon}{\delta}$ times a convex combination of ${\sf D}, {\sf E} \in \mathcal{D}$. That is, we have that ${\sf C} \in (1 + \frac{d \, \varepsilon}{\delta}) \, \mathcal{D}$ and hence $\mathcal{C} \subseteq (1 + \frac{d \, \varepsilon}{\delta}) \, \mathcal{D}$. We may also reverse the roles of $\mathcal{C}$ and $\mathcal{D}$ to complete the argument.
\end{proof}

\begin{prop}\label{prop:switch_dist}
Let ${\sf T} \in B(\cl H)^d$. Then $\beta({\sf T}) = 1$ implies that $\cW({\sf T}^{k\emph{-min}})$ converges to $\cW({\sf T})$ in the Hausdorff distance, and the converse holds if $0$ is in the interior of $\cW_1({\sf T})$. Similarly, $\gamma({\sf T}) = 1$ implies that $\cW({\sf T}^{k\emph{-max}})$ converges to $\cW({\sf T})$ in the Hausdorff distance, and the converse holds if $0$ is in the interior of $\cW_1({\sf T})$. 
\end{prop}
\begin{proof}
This follows immediately from Proposition \ref{prop:precise_dist}. In particular, note that a bound on $||{\sf T}||$ places a uniform bound $M$ on members of all of the matrix convex sets $\mathcal{W}({\sf T}^{k\text{-min}})$, $\mathcal{W}({\sf T})$, and $\mathcal{W}({\sf T}^{k\text{-max}})$. Similarly, these sets all have the same first level, so if $0$ is in the interior of $\mathcal{W}_1({\sf T})$, we may find $\delta > 0$ as in Proposition \ref{prop:precise_dist}.
\end{proof}

In the next section, we will connect the above constants to properties of the operator system ${\cl S}_{\sf T}$.

%%%%%%%%%%%%%%%%%%%%%%%%%%%%%%%%%%%%%%%%
%%%%%%%%%%%%%%%%%%%
\section{Scaling, Lifting Properties, and $1$-Exactness}\label{sec:scale_prop}

If ${\cl S}$ is an operator system, perhaps of infinite dimension, and $A$ is a unital $C^*$-algebra with an ideal $\cl I$ and associated quotient map $\pi: A \to A / \cl I$, then a UCP map
\[ \phi: \cl S \to A / \cl I \]
is said to admit a lift if there is a UCP map 
\[ \psi: \cl S \to A \]
such that $\phi = \pi \circ \psi$. Certainly not every UCP map into a quotient admits a lift, but for some operator systems, lifts may be obtained locally in the following sense. 

\begin{defn} \cite{KPTT}
An operator system $\cl S$ is said to have the {\bf operator system local lifting property}, or {\bf OSLLP}, if the following holds. For every unital $C^*$-algebra $A$, quotient $\pi: A \to A / \cl I$, and UCP map
\[ \phi: \cl S \to A / \cl I, \]
it follows that the restriction $\phi|_{\cl S_0}$ to any finite-dimensional operator subsystem ${\cl S}_0$ admits a lift. If $\cl S$ is itself finite dimensional, we simply say $\cl S$ has the {\bf lifting property}.
\end{defn}
\begin{remark}
 It is necessary to exercise some care when dealing with subsystems. The UCP maps in question take values in $A / {\cl I}$, not $B({\cl H})$, so there is no guarantee that an analogue of Arveson's extension theorem will hold a priori. 
\end{remark}

The Smith-Ward problem \cite{SW} was originally stated in the language of compact perturbations of essential matrix ranges. 
However, this problem is now known to be equivalent to the claim that all $3$-dimensional operator systems have the lifting property, by \cite[Theorem 11.5]{Ka-nuc}. In  \cite[Theorem~3.3]{Pa-1982}, a 5-dimensional operator system without the lifting property is constructed, and this example generates an infinite-dimensional $C^*$-algebra. Remarkably, \cite[Corollary~10.14]{Ka-nuc} shows that there is a 5-dimensional operator subsystem of the $4 \times 4$ matrices that does not possess the lifting property, namely 
\begin{equation}\label{eq:bad_matrix} \cl T_2 = \left\{ \begin{bmatrix} a & b & 0 & 0\\c & a & 0 & 0 \\0 & 0 & a & d \\ 0 & 0 & e & a \end{bmatrix} : a,b,c,d,e \in \bb C \right\}. \end{equation}

On the other hand, from \cite[Lemma~9.10]{Ka-nuc}, it is known that if $\cl S$ is finite-dimensional, then  $\text{OMAX}_k(\cl S)$ has the lifting property. We expand upon this result using $\Wmax{k}({\sf T})$ and the Hausdorff distance, or equivalently, the constant $\gamma({\sf T})$ for an appropriately chosen basis. This context extends and simplifies \cite{Pa-1982}. Similarly, it gives a matrix convex perspective to related results, such as \cite[Proposition 7.4]{Gol-Sing} and \cite[Lemmas 3.7 and 3.10]{Ozawa}.

\begin{thm}\label{thm:lifting_gamma} If ${\sf T}$ is a $d$-tuple of bounded operators, then items $(1)$-$(3)$ are equivalent. If $0$ is in the interior of $\mathcal{W}_1({\sf T})$, then they are also equivalent to $(3^\prime)$.
\begin{enumerate}
\item ${\cl S}_{\sf T}$ has the lifting property.
\item For every $\varepsilon > 0$ and UCP map $\phi: {\cl S}_{\sf T} \to B({\cl H}) / K({\cl H})$, there is a UCP map $\psi: {\cl S}_{\sf T} \to B({\cl H})$ such that $||\pi(\psi({\sf T})) - \phi({\sf T})|| < \varepsilon$.
\item $\mathcal{W}({\sf T}^{k\emph{-max}})$ converges to $\mathcal{W}({\sf T})$ in the Hausdorff distance as $k$ approaches $+\infty$.
\item [($3^{\prime}$)] $\gamma({\sf T}) = 1$.
\end{enumerate}
\end{thm}
\begin{proof}
The equivalence {\color{blue} $(3) \iff (3^\prime)$} under the assumption that $0$ is an interior point is exactly Proposition \ref{prop:switch_dist}. Items $(1)$-$(3)$ are unaffected by recoordinatization (including the case when the length of the tuple ${\sf T}$ is reduced), so we assume $0$ is an interior point throughout.

{\color{blue} $(1) \implies (2)$}. This is trivial.

{\color{blue} $(2) \implies (3)$}. Suppose $(2)$ holds but $(3)$ fails. Fix $\varepsilon > 0$ sufficiently small such that for any $k$, we may choose an element ${\sf M}^{(k)}$ of $\cW({\sf T}^{k\text{-max}})$ which is not within $\varepsilon$ of any point in $\cW({\sf T})$. 

Now, ${\sf M}^{(k)} \in \cW({\sf T}^{k\text{-max}})$, so by definition there is a UCP map sending
\[  {\sf T}^{k\text{-max}} \mapsto {\sf M}^{(k)}. \]
However, the sets $\cW({\sf T}^{k\text{-max}})$ are decreasing in $k$, so there is also a UCP map sending
\[ {\sf T}^{k\text{-max}} \mapsto {\sf M}^{(n)}, \hspace{.1 in} n \geq k, \]
and the direct sum gives a UCP map sending
\[ {\sf T}^{k\text{-max}} \mapsto \bigoplus\limits_{n=k}^\infty {\sf M}^{(n)}. \]
The lower index depends on $k$, so the largest direct sum $\bigoplus\limits_{n=1}^\infty {\sf M}^{(n)}$ might not be a UCP image of any ${\sf T}^{k\text{-max}}$. However, since each ${\sf M}^{(n)}$ is a matrix tuple, we may fix the target with a finite rank perturbation. Consequently, there exists a UCP map
\[ {\sf T}^{k\text{-max}} \mapsto \pi\left(\bigoplus\limits_{n=1}^\infty {\sf M}^{(n)}\right)  \]
where $\pi$ denotes the quotient map onto the Calkin Algebra. 

The sets $\cW({\sf T}^{k\text{-max}})$ are such that $\bigcap\limits_{k=1}^\infty \cW({\sf T}^{k\text{-max}}) = \cW({\sf T})$, so from the above we conclude there is a UCP map
\[ {\sf T} \mapsto \pi\left(\bigoplus\limits_{n=1}^\infty {\sf M}^{(n)}\right). \]
By (2), there is a UCP map ${\cl S}_{\sf T} \to B({\cl H})$ with ${\sf T} \mapsto {\sf L}$, where $||\pi({\sf L}) - \pi(\bigoplus\limits_{n=1}^\infty {\sf M}^{(n)})|| < \varepsilon/2$. Now, $\pi({\sf L} - \bigoplus\limits_{n=1}^\infty {\sf M}^{(n)})$ may be lifted to a tuple ${\sf E} \in B({\cl H})^d$ with $||{\sf E}|| < \varepsilon/2$, so we may write 
\[ {\sf L} = \bigoplus\limits_{n=1}^\infty {\sf M}^{(n)} + {\sf K} + {\sf E}\] 
where ${\sf K}$ is compact. The compressions of ${\sf L}$ to the (not necessarily reducing) subspaces corresponding to ${\sf M}^{(n)}$ produce elements ${\sf L}^{(n)}$ of $\cW({\sf L}) \subseteq \cW({\sf T})$ which have $\limsup\limits_{n \to \infty} ||{\sf L}^{(n)} - {\sf M}^{(n)}|| \leq \varepsilon/2 < \varepsilon$, a contradiction of assumption.

{\color{blue} $(3^\prime) \implies (1)$}. In \cite[Theorem~8.5]{KPTT}, it is shown that an operator system $\cl S$ has the OSLLP if and only if $\cl S \otimes_{\text{min}} B(\cl H) = \cl S \otimes_{\text{max}} B(\cl H)$ (that is, the identity map between these two operator systems is a complete order isomorphism). Further, \cite[Lemma~9.10]{Ka-nuc} shows that $\text{OMAX}_k(\cl S_{\sf T})$ has the lifting property.

Assume that $\gamma({\sf T}) = 1$. Let $\gamma_k := \gamma_k({\sf T})$, so that by definition we have $\cW({\sf T}) \subseteq \cW({\sf T}^{k\text{-max}}) \subseteq \gamma_k \,  \cW({\sf T})$, and there exist UCP maps $\phi_k$ and $\psi_k$ satisfying 

\[ \begin{array}{cccc} \phi_k: \cl S_{\sf T} \to \text{OMAX}_k(\cl S_{\sf T}) & & \psi_k: \text{OMAX}_k(\cl S_{\sf T}) \to \cl S_{\sf T} \\ \\ \phi_k(T_i) = \gamma_k^{-1} \, T^{k\text{-max}}_i & & \psi_k(T^{k\text{-max}}_i) = T_i \\ \\ \end{array} \]
for $1 \leq i \leq d$. Thus, for any Hilbert space $\cl H$, we have the following composition of UCP maps.
 \begin{multline*} \cl S_{\sf T} \otimes_{\text{min}} B(\cl H) \stackrel{\phi_k \otimes \text{id}}{\longrightarrow}
 \text{OMAX}_k(\cl S_{\sf T}) \otimes_{\text{min}} B(\cl H)  \stackrel{\text{id} \otimes \text{id}}{\longrightarrow} \\ \text{OMAX}_k(\cl S_{\sf T}) \otimes_{\text{max}} B(\cl H)  \stackrel{\psi_k \otimes \text{id}}{\longrightarrow} \cl S_{\sf T} \otimes_{\text{max}} B(\cl H)
 \end{multline*}

 The above shows that there is a UCP map 
\[ \zeta_k: \cl S_{\sf T} \otimes_{\text{min}} B(\cl H) \to \cl S_{\sf T} \otimes_{\text{max}} B(\cl H)\] 
sending $T_i \otimes A \mapsto \gamma_k^{-1} \, T_i \otimes A$ for all $k \in \mathbb{Z}^+$, $1 \leq i \leq d$, and $A \in B(\cl H)$.  Taking a limit in $k$ then gives that the identity map from $\cl S_{\sf T} \otimes_{\text{min}} B(\cl H)$ to $\cl S_{\sf T} \otimes_{\text{max}} B(\cl H)$ is UCP. As the identity map in the reverse direction is always UCP, we have that $\cl S_{\sf T} \otimes_{\text{min}} B(\cl H) = \cl S_{\sf T} \otimes_{\text{max}} B(\cl H)$ for any Hilbert space $\cl H$. Applying \cite[Theorem~8.5]{KPTT} yields that $\cl S_{\sf T}$ has the lifting property, as desired.
\end{proof}

Item (2) should be compared with the claim \cite[Proposition 7.4]{Gol-Sing}, which is given for $C^*$-algebras and approximations in a slightly different manner. See also \cite[Lemmas 3.7 and 3.10]{Ozawa}, which are stated for operator systems. We further note that the special role of the Calkin algebra in determining the lifting property is not a surprise, as in \cite[Proposition 3.13]{Ozawa}.

 If ${\sf A} \in M_n^d$ is a tuple of matrices, then the associated {\it free spectrahedron} is defined by

\[ D_{\sf A}(k) = \left\{ X \in M_k(\mathbb{C})^d: \text{Re}\left( \sum\limits_{j=1}^d X_j \otimes A_j \right) \leq I \right\}. \]
One may similarly define free spectrahedra in purely self-adjoint coordinates, with $\sum\limits_{j=1}^d \limits X_j \otimes A_j \leq I$. Modulo a minor headache from a change of sign (as $i^2 = -1$), these presentations are equivalent for tuples of non self-adjoint complex matrices. A free spectrahedron $D_{\sf A}$ is a closed matrix convex set, and it is bounded precisely when $0$ is an interior point of $\mathcal{W}_1({\sf A})$ by \cite[Lemma 3.4]{DDSS}. That result also shows that, in this case, $D_{\sf A}$ and $\mathcal{W}({\sf A})$ are mutually polar dual (see also \cite[Proposition 4.3]{Helton}). By following the polar dual, we see that an operator system $\mathcal{S}_{\sf T}$ with properly positioned coordinates has the lifting property precisely when $\mathcal{W}({\sf T})$ is well-approximated by free spectrahedra.

\begin{cor}\label{cor:lifting_finite}
Let ${\sf T} \in B({\cl H})^d$, and assume $0$ is an interior point of $\mathcal{W}_1({\sf T})$. Then $\mathcal{S}_{\sf T}$ has the lifting property if and only if for any $0 < \varepsilon < 1$, there is a free spectrahedron $D_{\sf A}$ such that $(1 - \varepsilon) D_{\sf A} \subseteq \mathcal{W}({\sf T}) \subseteq (1 + \varepsilon) D_{\sf A}$.
\end{cor}
\begin{proof}
From Theorem \ref{thm:lifting_gamma}, ${\cl S}_{\sf T}$ has the lifting property whenever $\gamma({\sf T}) = 1$, equivalently $\beta({\sf Q}) = 1$ for the dual system ${\cl S}_{\sf Q}$. However, each individual level $\mathcal{W}_k({\sf Q})$ may be approximated to arbitrary precision, in a scaling sense, by the convex hull of finitely many of its points. This produces an approximation of $\mathcal{W}({\sf Q}^{k\text{-min}})$ (from below) by $\mathcal{W}({\sf A})$ for a matrix tuple ${\sf A}$, hence an approximation for $\mathcal{W}({\sf T}^{k\text{-max}})$ (from above) by $D_{\sf A}$. The converse follows similarly.
\end{proof}

Note that since a free spectrahedron $D_{\sf A}$ has the property that for $c \not= 0$, $c \, D_{\sf A} = D_{c^{-1} {\sf A}}$, the approximation of $\mathcal{W}({\sf T})$ may be chosen from above or from below, as appropriate. Corollary \ref{cor:lifting_finite} should be compared with \cite[Proposition 3.9]{Kriel}. As a consequence of the above theorems, the operator system of (\ref{eq:bad_matrix}) corresponds to a matrix range with $\gamma({\sf T}) > 1$.

\begin{cor}\label{cor:matrix_units}
Let ${\sf T} = (E_{1,2}, E_{3, 4})$ where $E_{i, j}$ denotes a matrix unit in the $4 \times 4$ matrices. Then $\gamma(\sf T) > 1$.
\end{cor}
\begin{proof}
The operator system ${\cl S}_{\sf T}$ is precisely ${\cl T}_2$ as given in (\ref{eq:bad_matrix}), so it does not have the lifting property by \cite[Corollary 10.14]{Ka-nuc}. Note that $0$ is in the interior of the numerical range, so Theorem \ref{thm:lifting_gamma} shows that $\gamma({\sf T}) > 1$.
\end{proof}

We have not been able to compute $\gamma(\sf T)$ for ${\sf T}= (E_{1,2}, E_{3,4})$, and we believe that its precise value (or reasonably tight bounds) would be of interest.

Next, we turn to the concept of {\bf $1$-exactness} for operator systems. This concept originated in the work of Kirchberg \cite{Ki} in the setting of $C^*$-algebras and is related to how the minimal tensor product
behaves with respect to quotients.  A definition of $1$-exactness for operator systems was introduced in \cite[Definition 5.4]{KPTT}. It was subsequently studied in \cite{Ka-nuc}, where the more simplified term exact was used for the same concept (though we will not use this term).  Note that $1$-exactness of an operator system is detected solely from information about finite-dimensional subsystems \cite[Corollary 5.8]{KPTT}. Further, there is a tensor product of operator systems, called the {\bf el-tensor}, which detects $1$-exactness in the following sense.

\begin{thm}\cite[Theorem~5.7]{KPTT}. An operator system $\cl S$ is $1$-exact if and only if for every operator system $\cl T$, ${\cl S} \otimes_{\emph{min}} \cl T = \cl S \otimes_{\emph{el}} \cl T$.
\end{thm}

In the same vein as our results for the lifting property, $1$-exactness of a finite-dimensional operator system $\mathcal{S}_{\sf T}$ may be detected through convergence of $\cW({\sf T}^{k\text{-min}})$ to $\cW({\sf T})$, or equivalently by examining $\beta({\sf T})$ when the basis is appropriately positioned.

\begin{thm}\label{thm:exactness_beta} If ${\sf T} \in B(\cl H)^{d}$, then items $(1)$ and $(2)$ are equivalent. If $0$ is in the interior of $\mathcal{W}_1({\sf T})$, then they are also equivalent to $(2^\prime)$.
\begin{enumerate}
\item $\cl S_{\sf T}$ is $1$-exact.
\item $\mathcal{W}({\sf T}^{k\emph{-min}})$ converges to $\mathcal{W}({\sf T})$ in the Hausdorff distance as $k$ approaches $+\infty$.
\item [$(2^\prime)$] $\beta({\sf T}) = 1$.
\end{enumerate}
\end{thm}
\begin{proof}
First, consider that when $0$ is an interior point of $\mathcal{W}_1({\sf T})$, {\color{blue} $(2) \iff (2^\prime)$} follows from Proposition \ref{prop:switch_dist}. As before, $(1)$ and $(2)$ are not affected by recoordinatization, so we may assume $0$ is an interior point of $\mathcal{W}_1({\sf T})$ throughout.

{\color{blue} (1) $\implies (2^\prime)$}.  Since $0$ is in the interior of $\cW_1({\sf T})$, the polar dual of $\mathcal{W}({\sf T})$ is closed and bounded (where we use the self-adjoint polar dual instead if needed). If $\cW({\sf T})^\circ = \cW({\sf Q})$, then as in Proposition \ref{prop:polar_opsys_dual} and the comments thereafter, $\mathcal{S}_{\sf T}$ is completely isometrically isomorphic to the operator system dual of $\mathcal{S}_{\sf Q}$, and vice-versa by the bipolar theorem. Since $\mathcal{S}_{\sf T}$ is $1$-exact, \cite[Theorem 6.6]{Ka-nuc} shows that $\mathcal{S}_{\sf Q}$ has the lifting property. From Theorem \ref{thm:lifting_gamma}, we have $\gamma({\sf Q}) = 1$. Applying (\ref{eq:polar_dual_min_max}) gives $\beta({\sf T}) = 1$.

{\color{blue} $(2^\prime) \implies (1)$}. Assume $\beta({\sf T}) = 1$. By \cite[Lemma~9.8]{Ka-nuc}, the operator system $\text{OMIN}_k(\cl S_{\sf T})$ is $1$-exact for every $k$. Hence, given any operator system $\cl T$, we have that $\text{OMIN}_k(\cl S_{\sf T}) \otimes_{\text{min}} \cl T = \text{OMIN}_k(\cl S_{\sf T}) \otimes_{\text{el}} \cl T$, and we may follow a similar argument as for the lifting property. 

Let $\beta_k := \beta_k({\sf T})$, so that by definition $\cW({\sf T}^{k\text{-min}}) \subseteq \cW({\sf T}) \subseteq \beta_k \, \cW({\sf T}^{k\text{-min}})$, and there exist UCP maps $\phi_k$ and $\psi_k$ satisfying 

\[ \begin{array}{cccc} \phi_k: \text{OMIN}_k(\cl S_{\sf T}) \to \cl S_{\sf T}  & & \psi_k: \cl S_{\sf T} \to \text{OMIN}_k(\cl S_{\sf T}) \\ \\ \phi_k(T_i^{k\text{-min}}) = \beta_k^{-1} \, T_i & & \psi_k(T_i) = T^{k\text{-min}}_i\\ \\ \end{array} \]
for $1 \leq i \leq d$. Thus, for any operator system $\cl T$, we have the following composition of UCP maps.
 \begin{multline*} \cl S_{\sf T} \otimes_{\text{min}} {\cl T} \stackrel{\psi_k \otimes \text{id}}{\longrightarrow}
 \text{OMIN}_k(\cl S_{\sf T}) \otimes_{\text{min}} {\cl T}  \stackrel{\text{id} \otimes \text{id}}{\longrightarrow} \text{OMIN}_k(\cl S_{\sf T}) \otimes_{\text{el}} {\cl T}  \stackrel{\phi_k \otimes \text{id}}{\longrightarrow} \cl S_{\sf T} \otimes_{\text{el}} {\cl T}
 \end{multline*}

 The above shows that there is a UCP map
\[ \zeta_k: \cl S_{\sf T} \otimes_{\text{min}} {\cl T} \to \cl S_{\sf T} \otimes_{\text{el}} {\cl T}\] 
sending $T_i \otimes A \mapsto \beta_k^{-1} \, T_i \otimes A$ for all $k \in \mathbb{Z}^+$, $1 \leq i \leq d$, and $A \in {\cl T}$.  Taking a limit in $k$ shows that the identity map from $\cl S_{\sf T} \otimes_{\text{min}} {\cl T}$ to $\cl S_{\sf T} \otimes_{\text{el}} {\cl T}$ is UCP. As the identity map in the reverse direction is always UCP, we have that $\cl S_{\sf T} \otimes_{\text{min}} {\cl T} = \cl S_{\sf T} \otimes_{\text{el}} {\cl T}$ for any operator system ${\cl T}$. Finally, applying \cite[Theorem~5.7]{KPTT} gives that $\cl S_{\sf T}$ is $1$-exact.
\end{proof}

\begin{cor}\label{cor:both_alpha} If ${\sf T} \in B(\cl H)^d$, then items $(1)$ and $(2)$ are equivalent. If $0$ is in the interior of $\cl W_1(\sf T)$, then they are also equivalent to $(2^{\prime})$.
\begin{enumerate}
\item $\cl S_{\sf T}$ has the lifting property and is $1$-exact.
\item The Hausdorff distance between $\mathcal{W}({\sf T}^{k\emph{-min}})$ and $\mathcal{W}({\sf T}^{k\emph{-max}})$ approaches $0$ as $k$ approaches $+\infty$.
\item [$(2^{\prime})$] $\alpha(\sf T) =1$.
\end{enumerate}
\end{cor}

Following \cite[\S 8.2]{Lupini-limits} and \cite[\S 2]{Gol-Sing}, we note that an operator system is $1$-exact if and only if it is $1$-exact as an operator space \cite[Proposition 5.5]{KPTT}, and hence approximation results for operator spaces carry over to the operator system case. For a fixed dimension, the $1$-exact operator spaces are the closure of the matricial operator spaces in a cb Hausdorff distance, so for any finite-dimensional $1$-exact operator system ${\cl S}_{\sf T}$, there is a one-to-one, completely contractive map $\phi: {\cl S}_{\sf T} \to M_n$ with $||\phi^{-1}||_{cb}$ close to $1$. The map as stated is not guaranteed to be unital, but unitization results such as \cite[Lemma 8.6]{Lupini-limits} or \cite[Lemma 4.3]{Gol-Lup} allow one to unitize the map. For the specific claims we need, we find that it is easier to derive approximation properties from scaling conditions in matrix convex sets.

\begin{cor}\label{cor:exactness_finite}
If ${\sf T} \in B({\cl H})^d$, then items $(1)$-$(3)$ are equivalent. If $0$ is in the interior of $\mathcal{W}_1({\sf T})$, then they are also equivalent to $(3^\prime)$.
\begin{enumerate}
\item $\mathcal{S}_{\sf T}$ is $1$-exact.
\item For any $\varepsilon > 0$, there is a one-to-one UCP map $\phi$ from ${\cl S}_{\sf T}$ into some $M_n$ such that $||\phi^{-1}||_{cb} < 1 + \varepsilon$. 
\item For any $\varepsilon > 0$, there exists a matrix tuple ${\sf A}$ such that $\emph{dist}(\mathcal{W}({\sf T}), \mathcal{W}({\sf A})) < \varepsilon$.\item [$(3^\prime)$] For any $0 < \varepsilon < 1$, there exists a matrix tuple ${\sf A}$ such that $(1 - \varepsilon) \mathcal{W}({\sf A}) \subseteq \mathcal{W}({\sf T}) \subseteq (1 + \varepsilon) \mathcal{W}({\sf A})$.
\end{enumerate}
\end{cor}
\begin{proof}
If $0$ is in the interior of $\mathcal{W}_1({\sf T})$, then the equivalence ${\color{blue} (3) \iff (3^\prime)}$ follows from Proposition \ref{prop:precise_dist}. In particular, note that since Euclidean space is finite-dimensional, one may easily arrange for $0$ in the interior of $\mathcal{W}_1({\sf A})$ as well. Once again, we recoordinatize so that we may continue to use the interior assumption throughout.

{\color{blue} $(1) \implies (3^\prime)$}. We have that $\beta({\sf T}) = 1$, so given $\varepsilon > 0$, choose $k$ such that $\beta_k({\sf T}) < 1 + \varepsilon$. Now, $\mathcal{W}_k({\sf T})$ may be approximated (in a scaling sense) to arbitrary precision by the convex hull of finitely many points. The direct sum of these points produces a matrix tuple ${\sf A}$ such that
\[ (1 - \varepsilon) \mathcal{W}({\sf T}^{k\text{-min}}) \subseteq  \mathcal{W}({\sf A}) \subseteq \mathcal{W}({\sf T}^{k\text{-min}}) \]
and hence
\[ \mathcal{W}({\sf A}) \subseteq \mathcal{W}({\sf T}) \subseteq \cfrac{1 + \varepsilon}{1 - \varepsilon} \, \mathcal{W}({\sf A}). \]

{\color{blue} $(3^\prime) \implies (2)$}. We may rescale and suppose that some ${\sf A} \in \mathcal{W}_n({\sf T})$ has $\mathcal{W}({\sf A}) \subseteq \mathcal{W}({\sf T}) \subseteq (1 + \varepsilon) \mathcal{W}({\sf A})$. Consider the UCP map $\phi: \mathcal{S}_{\sf T} \to \mathcal{S}_{\sf A} \subseteq M_n$ defined by $\phi({\sf T}) = {\sf A}$. Similarly, consider the UCP map $\psi: {\cl S}_{\sf A} \to {\cl S}_{\sf T}$ defined by $\psi({\sf A}) = \frac{1}{1 + \varepsilon} {\sf T}$. The cb norm of $\phi^{-1}$ may be estimated by modifying the scale of the unit\footnote{We thank Adam Dor-On for showing us a similar method in a private communication.}. Namely, from our choice of coordinate system, there is a UCP map $\tau: \mathcal{S}_{\sf A} \to \mathbb{C}$ which annihilates ${\sf A}$. We then have that
\[ \phi^{-1} = (1 + \varepsilon) \psi - \varepsilon \, \tau \otimes I, \]
so $||\phi^{-1}||_{cb} \leq 1 + 2 \varepsilon$.

{\color{blue} $(2) \implies (1)$}.  A finite-dimensional operator space $X$ is $1$-exact by definition (see \cite[(17.4)' and p. 288]{Pisier}) precisely when the infimum of $||\phi||_{cb} \, ||\phi^{-1}||_{cb}$, ranging over all complete isomorphisms $\phi$ between $X$ and subspaces of matrix algebras, is $1$. In particular, since UCP maps are completely contractive, condition $(2)$ implies that $\mathcal{S}_{\sf T}$ is $1$-exact as an operator space. However, by \cite[Proposition 5.5]{KPTT}, we must also have that $\mathcal{S}_{\sf T}$ is $1$-exact as an operator system.
\end{proof}
\begin{remark}
Condition $(2)$ may suggest that $\mathcal{S}_{\sf T}$ should be completely order embedded into a nuclear unital $C^*$-algebra. This is not immediate, and results such as \cite[Corollary 18]{Ki-Wa} for separable nuclear operator systems suggest this might not be possible. One characterization of when an operator system can be embedded into $\mathcal{O}_2$ is given in \cite{Luthra} in terms of the $C^*$-envelope. However, we note that it is always possible to embed $\emph{OMIN}_k({\cl S}_{\sf T})$ into a (specific) nuclear $C^*$-algebra. Let
\[X_k = \{ \Phi \, \vert \, \Phi: \cl S \to M_k \,\, \emph{is UCP} \} \]
denote the {\it matrix state space} and let  
$C(X_k ; M_k) \equiv C(X_k) \otimes M_k$ 
be the $C^*$-algebra of continuous functions from $X_k$ into $M_k$. Further, let $\widehat{T_i}(k): X_k \to M_k$ be defined by
\[ \widehat{T_i}(k)(\Phi) = \Phi(T_i). \] 
It is easily seen that if we set $R_i = \bigoplus\limits_{k=1}^K \widehat{T_i}(k) \in \bigoplus\limits_{k=1}^K C(X_k ; M_k)$ and ${\sf R} = (R_1, \ldots, R_d)$, then
\[ {\cl W}({\sf T}^{\emph{k-min}}) = \cl W(\sf R).\]
Thus, there is a concrete embedding of $\emph{OMIN}_k(\cl S_{\sf T})$ into a nuclear $C^*$-algebra.

Contrast with \cite[Theorem 1.3]{Lupini-limits}, which shows that every separable $1$-exact operator system embeds into a single nuclear operator system, called the noncommutative Poulsen simplex. That result is an analogue of Kirchberg's result that every exact $C^*$-algebra embeds into $\cl O_2$. Note that in \cite{Lupini-limits}, $1$-exact operator systems are called \lq\lq exact".
\end{remark}

Similar to the case of free spectrahedra, since $c \, \mathcal{W}({\sf A}) = \mathcal{W}(c \, {\sf A})$, the approximation of $\mathcal{W}({\sf T})$ by $\mathcal{W}({\sf A})$ may be done from above or from below, as needed. Applying our results to the free group gives the following. 

 \begin{cor}\label{cor:free_unitaries} Fix $n \geq 2$, let $\bb F_n$ be the free group on $n$ generators, let $C^*(\bb F_n)$ be the full group $C^*$-algebra, and let $U_1, U_2, \ldots, U_n$ denote the  unitary operators corresponding to the generators of $\bb F_n$.  If we set ${\sf U} = (U_1, U_2, \ldots, U_n)$, then $\beta({\sf U}) > 1$.
 \end{cor}
 \begin{proof}
The operator system ${\cl S}_{\sf U}$ is not $1$-exact by \cite[Corollary 10.13]{Ka-nuc}, and $\cW_1({\sf U}) = {\overline{\mathbb{D}}}^n$ certainly has $0$ in the interior. By Theorem \ref{thm:exactness_beta}, we must have that $\beta({\sf U}) > 1$.
 \end{proof}

The operator system ${\cl S}_{\sf U} = {\cl S}_{\sf U^{[n]}}$ of $n$ universal unitaries, or rather the corresponding matrix convex set $\mathcal{W}({\sf U}^{[n]})$ of $n$-tuples of matrix contractions, has been of interest in previous scaling problems. However, most results about the universal unitaries have estimated $\beta_1({\sf U}^{[n]})$ through the (non-)existence of commuting normal dilations. For example, \cite[Theorem 4.4]{Pas_SSM} is equivalent to the claim $\beta_1({\sf U}^{[n]}) \leq \sqrt{2n}$, and a lower bound $\beta_1({\sf U}^{[n]}) \geq \sqrt{n}$ follows from the self-adjoint case \cite[Theorem 6.7]{Pas-Sh-So}. For $n = 2$, the lower bound $\beta_1({\sf U}^{[2]}) \geq \sqrt{2}$ is not optimal as a consequence of \cite[Theorem 6.3 and Figure 1]{Ge-Sh}, which show $\beta_1({\sf U}^{[2]}) \geq 1.543$. The interested reader is referred to \cite[Part 3]{Shalit-tour} for a summary of such first-level dilation results. These estimates have also been recently refined in \cite[\S 3]{Ge-unitaries} using the reduced $C^*$-algebra and the corresponding free Haar unitaries. In contrast, our result claims that for each $n \geq 2$,
\[ \lim\limits_{k \to \infty} \beta_k({\sf U}^{[n]}) > 1, \]
but again, we do not know this precise value.

The operator system ${\cl T}_2 = {\cl S}_{\sf T}, \, {\sf T}=(E_{1,2}, E_{3,4})$ in equation (\ref{eq:bad_matrix}) and Corollary \ref{cor:matrix_units}, which does not have the lifting property, is the dual of $\mathcal{S}_{\sf U^{[2]}}$. In fact, this duality is how the non-lifting property of ${\cl T}_2$ is derived in \cite[Corollary 10.14]{Ka-nuc}. The bases are positioned so that  the polar dual of $\cW({\sf U^{[2]}})$ is $\cW(2 {\sf T})$, so in particular we have that $\beta({\sf U^{[2]}})  = \gamma(2 {\sf T}) = \gamma({\sf T})$. In the next section, we see how this information can be presented using products of matrix convex sets and $k$-minimal/$k$-maximal operator system structures.

Since $1$-exactness or the lifting property of a finite-dimensional operator system is detected based on a Hausdorff distance approximation property for the matrix convex set, the corresponding collections of matrix convex sets are evidently closed. This is similar to known claims about the Hausdorff topology for operator spaces (or systems) in the cb norm.

\begin{prop}\label{prop:closedness_of_ex_li}
Fix $d \in \mathbb{Z}^+$, and equip the closed and bounded matrix convex subsets of $\mathbb{M}^d$ with the Hausdorff topology. If
\[ \mathcal{E} := \{\mathcal{W}({\sf T}): {\sf T} \in B({\cl H})^d, \, \mathcal{S}_{\sf T} \emph{ is } 1\emph{-exact}\} \]
and
\[ \mathcal{L} := \{\mathcal{W}({\sf T}): {\sf T} \in B({\cl H})^d, \, \mathcal{S}_{\sf T} \emph{ has the lifting property}\},\]
then $\mathcal{E}$ and $\mathcal{L}$ are closed. A similar claim holds for $\mathbb{M}^d_{sa}$ and tuples of self-adjoints.
\end{prop}
\begin{proof}
We will consider $1$-exactness, as the lifting property is similar. Suppose $\mathcal{C} = \cW({\sf T})$ is a closed and bounded matrix convex set in the closure of $\mathcal{E}$. Fix $\varepsilon > 0$, and let ${\cl S}_{\sf Q}$ be $1$-exact and such that $\text{dist}(\mathcal{W}({\sf Q}), \cW({\sf T})) < \varepsilon$. 

By Theorem \ref{thm:exactness_beta}, $\cW({\sf Q}^{k\text{-min}})$ converges to $\cW({\sf Q})$ in the Hausdorff distance, so choose $k$ large enough that any $Y \in \cW({\sf Q})$ is within $\varepsilon$ of a matrix convex combination of points in $\mathcal{W}_k({\sf Q})$, say, $\sum\limits_{j=1}^n V_j^* X^{[i]} V_j$ for $X^{[1]}, \ldots, X^{[n]} \in \mathcal{W}_k({\sf Q})$. Each $X^{[i]}$ may be approximated within $\varepsilon$ by some element of $\cW_k({\sf T})$. Since matrix convex combinations are contractive, we have that every $Y \in \cW({\sf Q})$ is within $2 \varepsilon$ of some point in $\cW({\sf T}^{k\text{-min}})$. Consequently, every point of $\cW({\sf T})$ is within $3 \varepsilon$ of a point of $\cW({\sf T}^{k\text{-min}})$. 

We conclude that the Hausdorff distance between $\Wmin{k}(\mathcal{C})$ and $\mathcal{C}$ is less than $3 \varepsilon$, hence $\mathcal{W}({\sf T}^{k\text{-min}})$ converges to $\mathcal{W}({\sf T})$ in the Hausdorff distance as $k \to \infty$. It follows from Theorem \ref{thm:exactness_beta} again that $\mathcal{S}_{\sf T}$ is $1$-exact, i.e. $\mathcal{W}({\sf T}) \in \mathcal{E}$. 
\end{proof}

Note that the collections $\mathcal{E}$ and $\mathcal{L}$ include sets $\mathcal{W}({\sf T})$ where the first level $\mathcal{W}_1({\sf T})$ lives  in a proper subspace of $\mathbb{C}^d$, and similarly $\mathbb{R}^d$ for the self-adjoint case. While the other proofs in this section would immediately recoordinatize this operator system, that reduction was possible only because a single coordinate change/projection did not alter the nature of the approximation being made. However, it is imperative in Proposition \ref{prop:closedness_of_ex_li} that degenerate sets be allowed as limits.

We close this section with some comments about the Smith-Ward problem  \cite{SW}, a problem which concerns a single operator $T \in B({\cl H})$. One formulation of the problem asks, if $T \in B({\cl H})$ and $\pi$ is the quotient map onto the Calkin algebra, is there a compact perturbation $T + K$ such that $\mathcal{W}(T + K) = \mathcal{W}(\pi(T))$?

\begin{cor} The following are equivalent.
\begin{enumerate}
\item The Smith-Ward problem has an affirmative answer.
\item For every $T \in B(\cl H)$, $\cl S_T$ has the lifting property.
\item For every $T \in B(\cl H)$, $\cl S_T$ is $1$-exact.
\item For every $T \in B(\cl H)$, the Hausdorff distance between $\cl W(T^{k\emph{-min}})$ and $\cl W(T^{k\emph{-max}})$ approaches $0$ as $k$ approaches $+\infty$.
\item For every $T \in B(\cl H)$, $\cl W(T^{k\emph{-min}})$ converges to $\cl W(T)$ in the Hausdorff distance as $k$ approaches $+\infty$.
\item For every $T \in B(\cl H)$, $\cl W(T^{k\emph{-max}})$ converges to $\cl W(T)$ as $k$ approaches $+\infty$.
\item For every $T \in B(\cl H)$ with $0$ in the interior of $\cl W_1(T)$, $\alpha(T) =1$.
\item For every $T \in B(\cl H)$ with $0$ in the interior of $\cl W_1(T)$, $\beta(T) =1$.
\item For every $T \in B(\cl H)$ with $0$ in the interior of $\cl W_1(T)$, $\gamma(T) =1$.
\end{enumerate}
\end{cor}
\begin{proof} The equivalence of (1), (2), and (3) is given by \cite[Theorem~11.5]{Ka-nuc} and the comments immediately thereafter. In particular, note that if the dimension of $\mathcal{S}_{\sf T}$ is less than $3$, the statement may be directly verified with a recoordinatization and a reduction to the self-adjoint case. Equivalently, one only needs to consider $T$ such that $0$ is an interior point of ${\cl W}_1(T) \subseteq \mathbb{C}$. The remaining equivalences follow from Theorems \ref{thm:lifting_gamma} and \ref{thm:exactness_beta} as well as Corollary \ref{cor:both_alpha}.
\end{proof}

The equivalence of (1) and (6) also follows from \cite[Theorem~3.15]{Pa-1982}.

%%%%%%%%%%%%%%%%%
%%%%%%%%
%%%%%%%   
\section{Sums and products}\label{sec:products}

Given two operator systems $\cl S$ and $\cl T$, there are two natural ways to form a direct sum operator system, denoted $\oplus$ and $\oplus_1$.  If $\cl S \subseteq B(\cl H)$ and $\cl T \subseteq B(\cl K),$ then $\cl S \oplus \cl T$ is the quite natural operator system on $B(\cl H \oplus \cl K)$ given by
\[ \cl S \oplus \cl T = \left\{ \begin{pmatrix} S & 0 \\ 0 & T \end{pmatrix} \in B(\cl H \oplus \cl K) : S \in \cl S, \, T \in \cl T \right\}.\]
However, the units of $\cl S$ and $\cl T$ are distinct from the unit of $\cl S \oplus \cl T$. If $\phi: \cl S \oplus \cl T \to B(\cl H^{\prime})$ is a UCP map, then all that can be said is that $\phi(I_{\cl S} \oplus 0) = P \ge 0$ and $\phi(0 \oplus I_{\cl T}) = Q \ge 0$ with $P+Q= I_{\cl H^{\prime}}$. 

It is easy to check that if $\phi: \cl S \to B(\cl H)$ is a completely positive map with $\phi(I) = P$, then there exists a UCP map $\phi_1: \cl S \to B(\cl H)$ and an operator $A \in B(\cl H)$ with $A^*A=P$ such that $\phi(S) = A^*\phi_1(S) A$.  Thus, every UCP map $\gamma: \cl S \oplus \cl T \to B(\cl H)$ has the form
\[ \gamma( S \oplus T) = A^* \phi_1(S) A + B^* \psi_1(T)B,\]
where $A, B \in B(\cl H)$ satisfy $A^*A+B^*B= I$, and both $\phi_1: \cl S \to B(\cl H)$ and $\psi_1: \cl T \to B(\cl H)$ are UCP maps.

Now, if ${\sf T} = (T_1,\ldots,T_d)$ and ${\sf R}=(R_1,\ldots, R_k)$, then the corresponding direct sum in $\cl S_{\sf T} \oplus \cl S_{\sf R}$ is a $(d+k)$-tuple of operators, which we denote as
\begin{equation} \label{eq:direct_sum_op} {\sf T}  \, \smboxplusraised \, {\sf R} :=(T_1 \oplus 0,\ldots, T_d \oplus 0, 0 \oplus R_1, \ldots, 0 \oplus R_k). \end{equation}

The above facts then show that
\[ \mathcal{W}_n({\sf T} \, \smboxplusraised \, {\sf R}) = \{ (A^* \, {\sf X} \, A, B^* \, {\sf Y} \, B): A,B \in M_n, \, A^*A+B^*B = I_n, \, {\sf X} \in \cW_n({\sf T}), {\sf Y} \in \cW_n({\sf T})\}. \]
Equivalently, $\mathcal{W}({\sf T} \, \smboxplusraised \, {\sf R})$ is the matrix convex hull of tuples $({\sf X}, 0)$ and $(0, {\sf Y})$ for ${\sf X} \in \mathcal{W}({\sf T})$ and ${\sf Y} \in \mathcal{W}({\sf T})$, where the tuples of zeroes are of the appropriate matrix dimensions and lengths.

The notion of operator system direct sum, $\oplus_1$, was developed in  \cite{KPTT}. In particular, in ${\cl S} \oplus_1 {\cl T}$, the unit is simultaneously equal to the units of ${\cl S}$ and ${\cl T}$. The key properties of this system that we shall use are that there exist UCP inclusion maps, $\iota_1: \cl S \to \cl S \oplus_1 \cl T$ and $\iota_2: \cl T \to \cl S \oplus_1 \cl T$ such that  $\gamma: \cl S \oplus_1 \cl T \to B(\cl H)$ is UCP if and only if there exist UCP maps $\phi: \cl S \to B(\cl H)$ and $\psi: \cl T \to B(\cl H)$ such that $\gamma(\iota_1(S) +  \iota_2(T)) = \phi(S) + \psi(T)$. The operator system was defined as a quotient using the abstract theory of operator systems, so even if one is given concrete representations of $\cl S$ and $\cl T$ on Hilbert spaces, it is not so clear how to represent $\cl S \oplus_1 \cl T$ as operators on a Hilbert space. Consider, for example, the $\oplus_1$ of a universal unitary operator system with itself $n$ times, as in Corollary \ref{cor:free_unitaries}.  However, this construction still has many nice properties.

Given ${\sf T}= (T_1,\ldots,T_d)$ and ${\sf R}=(R_1,\ldots,R_k)$ and corresponding operator systems $\cl S_{\sf T}$ and $\cl S_{\sf R}$, we set
\[ {\sf T} \, \smboxplusraised_1 \, {\sf R} = ( \iota_1(T_1), \ldots, \iota_1(T_d), \iota_2(R_1),\ldots, \iota_2(R_k)),\] where the $(d+k)$-tuple is inside the operator system $\cl S_{\sf T} \oplus_1 \cl S_{\sf R}$. The operator system generated by ${\sf T} \, \smboxplusraised_1 \, {\sf R}$ is unitally completely order isomorphic to $\cl S_{\sf T} \oplus_1 \cl \cl S_{R}$.
It follows that the $n$th level of the matrix range is a Cartesian product,
\[ \cl W_n({\sf T} \, \smboxplus_1 \, {\sf R}) = \{({\sf X}, {\sf Y}): {\sf X} \in \cW_n({\sf T}), {\sf Y} \in \cW_n({\sf R})\}, \]
and hence the matrix range is a levelwise Cartesian product.

\begin{defn} Given matrix convex sets $\cl C = \cup_n \cl C_n$ on $\bb C^d$ and $\cl D = \cup_n \cl D_n$ on $\bb C^k$, define the matrix convex set $\mathcal{C} \times \mathcal{D}$ levelwise by
\[ (\mathcal{C} \times \mathcal{D})_n :=  \{({\sf X}, {\sf Y}): {\sf X} \in \mathcal{C}_n, {\sf Y} \in \mathcal{D}_n\}. \]

\vspace{.1 in}

\noindent Similarly, let $\mathcal{C} \times_1 \mathcal{D}$ denote the matrix convex hull of all tuples $({\sf X}, {\sf 0})$ and $({\sf 0}, {\sf Y})$ where ${\sf X} \in \mathcal{C}$ and ${\sf Y} \in \mathcal{D}$. Equivalently, for each $n$,
\[ (\mathcal{C} \times_1 \mathcal{D})_n := \{(A^* {\sf X} A, B^* {\sf Y} B): A, B \in M_n, A^*A + B^*B = I_n, {\sf X} \in \mathcal{C}_n, {\sf Y} \in \mathcal{D}_n\}. \]
\end{defn}

\vspace{.2 in}

Based on the previous discussion, we have that
\[ \cl W({\sf T} \, \smboxplus \, {\sf R}) = \cl W({\sf T}) \times_1 \cl W({\sf R})\]
and
\[ \cl W({\sf T} \, \smboxplus_1 \, {\sf R}) = \cl W({\sf T}) \times \cl W({\sf R})\]
whenever ${\sf T}$ and ${\sf R}$ are tuples of bounded operators. In particular, the examples of \cite[\S 10]{Ka-nuc} fit nicely in this presentation.

\begin{exam}
Consider twice the $2 \times 2$ matrix unit, $Z = 2 E_{1,2} = \begin{pmatrix} 0 & 2 \\ 0 & 0 \end{pmatrix}$, so that $\mathcal{W}(Z) = \Wmax{1}(\overline{\mathbb{D}})$. This is a consequence of Ando's characterization \cite{An} of operators of numerical radius 1; it is alternatively the polar dual of \cite[Proposition 14.14]{HKMS}. Then
\[ \cW(Z \, \smboxplus \, Z) =  \cW(2 \, E_{1,2}, 2 \, E_{3,4}) = \Wmax{1}(\overline{\mathbb{D}}) \times_1 \Wmax{1}(\overline{\mathbb{D}}), \]
where we note that $Z \smboxplus Z$ is a pair of $4 \times 4$ matrices. From Corollary \ref{cor:matrix_units}, we have that $\gamma(Z \, \smboxplus \, Z) > 1$, even though $\gamma(Z) = 1$ (as $\cW(Z)$ is a $\Wmax{1}$-set). 
\end{exam} 

The above may be extended to more than two summands, and a similar decomposition occurs for the universal unitaries.

\begin{exam}
Let $\sf U^{[n]}$ be the $n$-tuple which contains the unitary generators of $C^*(\mathbb{F}_n)$. Then ${\sf U}^{[n]}$ may be viewed as
\[ {\sf U}^{[n]} = N \, \smboxplusraised_1 \, N \, \smboxplusraised_1 \, \cdots \, \smboxplusraised_1 \, N \]
where $N = {\sf U}^{[1]}$ is a single unitary operator with $\sigma(N) = \mathbb{S}^1$. From \cite[Theorem 2.7 and Corollary 2.8]{DDSS} and the above, we see that
\[ \cW({\sf U}^{[n]}) = \Wmin{1}(\overline{\mathbb{D}}) \times \Wmin{1}(\overline{\mathbb{D}}) \times \cdots \times \Wmin{1}(\overline{\mathbb{D}}). \]
From Corollary \ref{cor:free_unitaries}, we have that $\beta({\sf U}^{[n]}) > 1$ for $n \geq 2$, even though $\beta({\sf U}^{[1]}) = 1$ (as $\cW({\sf U}^{[1]})$ is a $\Wmin{1}$-set). 
\end{exam}

Note that the product on the right hand side is precisely the one studied in \cite{Pas_SSM} in order to estimate 
\[ \sqrt{n} \leq \beta_1({\sf U}^{[n]}) \leq \sqrt{2n}. \]
The operator systems above are dual to each other, and similarly the matrix convex sets are polar duals. It is therefore not surprising that the polar dual switches the two product operations.

\begin{prop}
Let $\mathcal{C}$ and $\mathcal{D}$ be closed and bounded matrix convex sets over complex Euclidean space, with $0$ in the interior of $\mathcal{C}_1$ and $\mathcal{D}_1$. Then

\[ (\mathcal{C} \times \mathcal{D})^\circ = \mathcal{C}^\circ \times_1 \mathcal{D}^\circ\]
and
\[ (\mathcal{C} \times_1 \mathcal{D})^\circ = \mathcal{C}^\circ  \times \mathcal{D}^\circ. \]
If instead $\mathcal{C}$ and $\mathcal{D}$ are closed and bounded matrix convex sets of self-adjoints, with $0$ in the interior of $\mathcal{C}_1$ and $\mathcal{D}_1$ when viewed as subsets of real Euclidean space, then analogous claims hold for the self-adjoint polar dual.
\end{prop}
\begin{proof}
 First, note that the $\times$ and $\times_1$ operations preserve the inclusion of the zero tuple (specifically, as an interior point). On account of the bipolar theorem \cite[Corollary 5.5]{EF}, it suffices to prove one equality. Let $\mathcal{C}$ consist of $n$-tuples and let $\mathcal{D}$ consist of $d$-tuples, and consider that $(\mathcal{C} \times_1 \mathcal{D})^\circ$ contains precisely the tuples ${\sf A} = (A_1, \ldots, A_{n+d})$ such that for all $X \in \mathcal{C} \times_1 \mathcal{D}$,
\begin{equation}\label{eq:polarlong} \text{Re}\left( \sum\limits_{k=1}^{n+d} A_k \otimes X_k \right) \leq I. \end{equation}
Now, $\mathcal{C} \times_1 \mathcal{D}$ is the matrix convex hull of tuples of the form $({\sf C}, {\sf 0})$ and $({\sf 0}, {\sf D})$, and the inequality in (\ref{eq:polarlong}) is preserved under matrix convex combinations. Thus, $X \in (\mathcal{C} \times_1 \mathcal{D})^\circ$ if and only if for all ${\sf C} \in \mathcal{C}$ and ${\sf D} \in \mathcal{D}$,
\[ \text{Re}\left( \sum\limits_{k=1}^n A_k \otimes C_k \right) \leq I \,\,\,\,\,\, \text{ and } \,\,\,\,\,\, \text{Re}\left( \sum\limits_{k=1}^d A_{n+k} \otimes D_k \right) \leq I. \]
This is equivalent to the claim that $(A_1, \ldots, A_n) \in \mathcal{C}^\circ$ and $(A_{n+1}, \ldots, A_{n+d}) \in \mathcal{D}^\circ$. That is, ${\sf A} \in \mathcal{C}^\circ \times \mathcal{D}^\circ$. The self-adjoint case is similar.
\end{proof}

Similarly, the $k$-min and $k$-max operations factor through certain products.

\begin{prop}
If $\mathcal{C}$ and $\mathcal{D}$ are closed and bounded matrix convex sets over Euclidean space, then 
\[ \Wmin{k}( \mathcal{C} \times_1 \mathcal{D}) = \Wmin{k}(\mathcal{C}) \times_1 \Wmin{k}(\mathcal{D})\]
and
\[ \Wmax{k}( \mathcal{C} \times \mathcal{D}) = \Wmax{k}(\mathcal{C}) \times \Wmax{k}(\mathcal{D}). \]
\end{prop}
\begin{proof}
The proof for $\Wmax{k}$ is trivial. For $\Wmin{k}$, it is immediate that the right hand side is the matrix convex hull of tuples $({\sf X}, {\sf 0})$ and $({\sf 0}, {\sf Y})$ where $X \in \mathcal{C}_k$ and $Y \in \mathcal{D}_k$, and hence it is contained in the left hand side. However, an arbitrary element of $\Wmin{k}( \mathcal{C} \times_1 \mathcal{D})$ is a matrix convex combination of points in its $k$th level, each of which is by definition a matrix convex combinations of two tuples $({\sf X}, 0)$ and $(0, {\sf Y})$ for $X \in \mathcal{C}_k$ and $Y \in \mathcal{D}_k$. Therefore the other containment also holds.
\end{proof}

\begin{prop}
Let $\mathcal{C}$ and $\mathcal{D}$ be closed and bounded matrix convex sets over Euclidean space. Then
\[ \beta_k(\mathcal{C} \times_1 \mathcal{D}) = \max(\beta_k(\mathcal{C}), \beta_k(\mathcal{D}))\]
and
\[ \gamma_k(\mathcal{C} \times \mathcal{D}) = \max(\gamma_k(\mathcal{C}), \gamma_k(\mathcal{D})). \]
Similarly, if ${\sf T}$ and ${\sf R}$ are tuples of operators on Hilbert space, then
\[ \beta_k(\sf T \, \smboxplus \, \sf R) = \max(\beta_k(\sf T), \beta_k(\sf R)) \]
and
\[ \gamma_k(\sf T \, \smboxplus_1 \, \sf R) = \max(\gamma_k(\sf T), \gamma_k(\sf R)). \]
\end{prop}

Therefore, our results are consistent with the fact that for finite-dimensional operator systems, $\oplus$ (corresponding to $\times_1$) preserves $1$-exactness and $\oplus_1$ (corresponding to $\times$) preserves the lifting property. However, we note that as in \cite[Corollary 10.13]{Ka-nuc} and \cite[Corollary 10.14]{Ka-nuc}, $\oplus_1$ does not preserve $1$-exactness and $\oplus$ does not preserve the lifting property.

\section*{Acknowledgments}

VP is partially supported by NSERC Discovery Grant 03784.


\begin{thebibliography}{99}

\bibitem{An} Ando, Tsuyoshi. Structure of operators with numerical radius one, {\it Acta Sci. Math. (Szeged)} 34 (1973), 11-15.

\bibitem{Arv_sub}  Arveson, William B. On subalgebras of $C^*$-algebras. {\it Bull. Amer. Math. Soc.} 75 (1969), 790-794. 

\bibitem{DDSS}  Davidson, Kenneth R.; Dor-On, Adam; Shalit, Orr Moshe; and Solel, Baruch. Dilations, inclusions of matrix convex sets, and completely positive maps. {\it Int. Math. Res. Not.} 13 (2017), 4069-4130. Corrected version in arXiv:1601.07993.

\bibitem{Adam-thesis} Dor-On, Adam. Techniques in operator algebras: classification, dilation and non-commutative boundary theory. Dissertation, University of Waterloo, 2017. http://hdl.handle.net/10012/12131.

\bibitem{EF} Effros, Edward G. and Winkler, Soren. Matrix convexity: operator analogues of the bipolar and Hahn-Banach theorems, {\it J. Funct. Anal.} 144 (1997), 117-152.

\bibitem{Ge-unitaries} Gerhold, Malte; Pandey, Satish L.; Shalit, Orr; and Solel, Baruch. Dilations of unitary tuples. arXiv:2006.01869.

\bibitem{Ge-Sh} Gerhold, Malte and Shalit, Orr. Dilations of $q$-commuting unitaries. arXiv:1902.10362. To appear in {\it Int. Math. Res. Not.} doi:10.1093/imrn/rnaa093.

\bibitem{Ge-random} Gerhold, Malte and Shalit, Orr. On the matrix range of random matrices. arXiv:1911.12102. To appear in {\it J. Operator Theory.}

\bibitem{Gol-Lup} Goldbring, Isaac and Lupini, Martino. Model-theoretic aspects of the Gurarij operator system. {\it Israel J. Math.} 226 (2018), 87-118.

\bibitem{Gol-Sing}  Goldbring, Isaac and Sinclair, Thomas. Omitting types in operator systems. {\it Indiana Univ. Math. J.} 66 (2017), no. 3, 821-844.

\bibitem{Helton} Helton, J. William; Klep, Igor; and McCullough, Scott. The tracial Hahn-Banach theorem, polar duals, matrix convex sets, and projections of free spectrahedra.  {\it J. Eur. Math. Soc.} 19 (2017), no. 6, 1845-1897.

\bibitem{HKMS} Helton, J. William; Klep, Igor; McCullough, Scott; and Schweighofer, Markus. Dilations, linear matrix inequalities, the matrix cube problem and beta distributions. {\it Mem. Amer. Math. Soc.} 257 (2019), no. 1232, vi+106 pp.

\bibitem{Ka-nuc} Kavruk, Ali S., Nuclearity related properties in operator systems. {\it J. Operator Theory} 71 (2014), no. 1, 95-156.

 \bibitem{KPTT} Kavruk, Ali S.; Paulsen, Vern I.; Todorov, Ivan G.; and Tomforde, Mark.  Quotients, exactness, and nuclearity in the operator system category, {\it Adv. Math.} 235 (2013), 321-360. 

\bibitem{Ki} Kirchberg, Eberhard. On non-semisplit extensions, tensor products and exactness of group $C^*$-algebras, {\it Invent. Math.} 112 (1993) 449-489. 

\bibitem{Ki-Wa} Kirchberg, Eberhard and Wassermann, Simon. $C^*$-algebras generated by operator systems, {\it J. Funct. Anal.} 155 (1998) 324-351.

\bibitem{Kriel} Kriel, Tom-Lukas. An introduction to matrix convex sets and free spectrahedra. {\it Complex Anal. Oper. Theory} 13 (2019), 3251-3335.

\bibitem{Li_essMR} Li, Chi-Kwong; Paulsen, Vern I.; and Poon, Yiu-Tung. Preservation of the joint essential matricial range. {\it Bull. Lond. Math. Soc.} 51 (2019), no. 5, 868-876.

\bibitem{Lupini-limits} Lupini, Martino, Fra{\"{i}}ss{\'{e}} limits in functional analysis. {\it Adv. Math.} 338 (2018) 93-174.

\bibitem{Luthra} Luthra, Preeti and Kumar, Ajay. Embeddings and $C^*$-envelopes of exact operator systems. {\it Bull. Aust. Math. Soc.}, 96 (2017), no. 2, 274-285.

\bibitem{Ozawa} Ozawa, Narutaka. About the QWEP conjecture. {\it Int. J. Math.} 15 (2004), no. 5, 501-530.

\bibitem{Pas_SSM} Passer, Benjamin. Shape, scale, and minimality of matrix ranges. {\it Trans. Amer. Math. Soc.} 372 (2019), no. 2, 1451-1484.
 
\bibitem{Pas-Sh-So} Passer, Benjamin; Shalit, Orr Moshe; and Solel, Baruch. Minimal and maximal matrix convex sets. {\it J. Funct. Anal.} 274 (2018), no. 11, 3197-3253.

\bibitem{Pa-1982} Paulsen, Vern I. Preservation of essential matrix ranges by compact perturbations. {\it J. Operator Theory} 8 (1982), no. 2, 299-317.

\bibitem{Pisier} Pisier, Gilles. Introduction to operator space theory. {\it London Mathematical Society Lecture Note Series}, 294. Cambridge University Press, Cambridge, 2003. viii+478 pp. ISBN: 0-521-81165-1.

\bibitem{Shalit-tour} Shalit, Orr. Dilation theory: a guided tour. arXiv:2002.05596. To appear in {\it Oper. Theory
  Adv. Appl.}, Birkh\"auser.

\bibitem{SW} Smith, Roger R. and Ward, Joseph D., Matrix ranges for Hilbert space operators, {\it Amer. J. Math.} 102 (1980), 1031-1081.


\bibitem{Xh-phd} Xhabli, Blerina. Universal operator system structures on ordered spaces and their applications. Dissertation, University of Houston, 2009. 122 pp. ISBN: 978-1109-70014-5, ProQuest LLC.


\bibitem{Xh-JFA} Xhabli, Blerina. The super operator system structures and their applications in quantum entanglement theory. {\it J. Funct. Anal.} 262 (2012), no. 4, 1466-1497.


\end{thebibliography}
\end{document}